\documentclass[12pt]{article}
\usepackage{latexsym, amssymb, amsthm, amsmath}
\usepackage{dsfont}

\DeclareMathAlphabet\gothic{U}{euf}{m}{n}

\makeatletter
\def\eqnarray{\stepcounter{equation}\let\@currentlabel=\theequation
\global\@eqnswtrue
\tabskip\@centering\let\\=\@eqncr
$$\halign to \displaywidth\bgroup\hfil\global\@eqcnt\z@
  $\displaystyle\tabskip\z@{##}$&\global\@eqcnt\@ne
  \hfil$\displaystyle{{}##{}}$\hfil
  &\global\@eqcnt\tw@ $\displaystyle{##}$\hfil
  \tabskip\@centering&\llap{##}\tabskip\z@\cr}

\def\endeqnarray{\@@eqncr\egroup
      \global\advance\c@equation\m@ne$$\global\@ignoretrue}

\def\@yeqncr{\@ifnextchar [{\@xeqncr}{\@xeqncr[5pt]}}
\makeatother

\allowdisplaybreaks[1]

\newtheorem{lemm}{Lemma}[section]
\newtheorem{thrm}[lemm]{Theorem}

\newtheorem{eeg}[lemm]{Example}
\newtheorem{rrema}[lemm]{Remark}
\newtheorem{prop}[lemm]{Proposition}
\newtheorem{ddefi}[lemm]{Definition}

\newenvironment{defi}{\begin{ddefi} \rm}{\end{ddefi}}

\newcounter{teller}

\newenvironment{tabel}{\begin{list}%
{\rm  (\alph{teller})\hfill}{\usecounter{teller} \leftmargin=1.1cm
\labelwidth=1.1cm \labelsep=0cm \parsep=0cm}
                      }{\end{list}}

\newcounter{tellerr}

\newenvironment{tabeleq}{\begin{list}%
{\rm  (\roman{tellerr})\hfill}{\usecounter{tellerr} \leftmargin=1.1cm
\labelwidth=1.1cm \labelsep=0cm \parsep=0cm}
                         }{\end{list}}

\newcounter{tellerrr}

\newcounter{proofstep}

\newcommand{\Ni}{\mathds{N}}

\newcommand{\Ri}{\mathds{R}}
\newcommand{\Ci}{\mathds{C}}

\newcommand{\R}{\mathrm{Re} \,}
\newcommand{\I}{\mathrm{Im} \,}
\newcommand{\tr}{\mathrm{tr} \,}

\newcommand{\D}{\partial}

\newcommand{\dist}{\mathrm{dist} \,}

\newcommand{\one}{\mathds{1}}

\newlength{\hightcharacter}
\newlength{\widthcharacter}

\begin{document}

\thispagestyle{empty}

\vspace*{1cm}
\begin{center}
{\Large\bf On sectoriality of degenerate elliptic operators} \\[5mm]
\large Tan Duc Do \\[10mm]

\end{center}

\vspace{5mm}

\begin{center}
{\bf Abstract}
\end{center}

\begin{list}{}{\leftmargin=1.8cm \rightmargin=1.8cm \listparindent=10mm 
   \parsep=0pt}
\item
Let $c_{kl} \in W^{1,\infty}(\Omega, \Ci)$ for all $k,l \in \{1, \ldots, d\}$ and $\Omega \subset \Ri^d$ be open with Lipschitz boundary.
We consider the divergence form operator 
$
A_p = - \sum_{k,l=1}^d \D_l (c_{kl} \, \D_k)
$
in $L_p(\Omega)$ when the coefficient matrix satisfies $(C(x) \, \xi, \xi) \in \Sigma_\theta$ for all $x \in \Omega$ and $\xi \in \Ci^d$, where $\Sigma_\theta$ be the sector with vertex 0 and semi-angle $\theta$ in the complex plane.
We show that a sectorial estimate hold for $A_p$ for all $p$ in a suitable range.
We then apply these estimates to prove that the closure of $-A_p$ generates a holomorphic semigroup under further assumptions on the coefficients.
The contractivity and consistency properties of these holomorphic semigroups are also considered.
\end{list}

\vspace{2.5cm}
\noindent
November 2016

\vspace{5mm}
\noindent
AMS Subject Classification: 35J25, 35K65, 47B44.

\vspace{5mm}
\noindent
Keywords: Degenerate elliptic operator, sectorial operator, 
holomorphic semigroup, contraction semigroup, consistent semigroup.

\vspace{10mm}

\noindent
{\bf Home institution:}    \\[3mm]
Department of Mathematics \\
University of Auckland  \\ 
Private bag 92019 \\ 
Auckland 1142 \\
New Zealand  \\[5pt]
Email: tan.do@auckland.ac.nz

\newpage

\setcounter{page}{1}

\section{Introduction} \label{Intro}

In his book \cite{Kat1} Kato showed that an $m$-sectorial operator in a Hilbert space generates a (quasi-)contraction holomorphic semigroup.
One can generalise the notion of sectorial operators to $L_p$-spaces as follows (cf.\ \cite[Definition 1.5.8]{Gol}, \cite[Subsection V.3.10]{Kat1} and \cite[Definition 1]{CiaM}).

\begin{defi}
Let $d \in \Ni$, $\Omega \subset \Ri^d$ be open and $p \in (1, \infty)$.
Let $A_p$ be an operator in $L_p(\Omega)$.
Then $A_p$ is said to be \emph{sectorial} if there exists a $K > 0$ such that
\begin{equation} \label{sectorial estimate}
|\I (A_p u, |u|^{p-2} \, u \, \one_{[u \neq 0]})| \leq K \, \R (A_p u, |u|^{p-2} \, u \, \one_{[u \neq 0]})
\end{equation}
for all $u \in D(A_p)$.
\end{defi}

There are certain interests in showing that an operator is sectorial in this generalised sense.
The significance of these estimates lies in the fact that they are useful in showing that the operators under consideration satisfy a necessary condition to generate holomorphic contraction semigroups.
In particular the estimate \eqref{sectorial estimate} can be established for certain second-order differential operators in divergence form.
In the proof of \cite[Theorem 7.3.6]{Paz}, Pazy showed that \eqref{sectorial estimate} holds when the operator is strongly elliptic with symmetric real-valued $C^1$-coefficients, with an explicit constant $K$ which depends on the coefficients, the elliticity constant and $p$.
Okazawa improved Pazy's result and showed that the estimate also holds for degenerate elliptic operators with symmetric real-valued $C^1$-coefficients, with $K = \frac{|p-2|}{2 \, \sqrt{p-1}}$ (cf.\ \cite{Oka2}).
Ouhabaz in \cite[Theorem 3.9]{Ouh5} proved that \eqref{sectorial estimate} is true for generators of sub-Markovian semigroups.
It is interesting to note that \cite[Theorem 3.9]{Ouh5} gives the same constant $K$ in \eqref{sectorial estimate} as in \cite{Oka2}.

In this paper we will prove the sectorial estimate \eqref{sectorial estimate} for degenerate elliptic second-order differential operators with bounded complex-valued coefficients.
The results are generalisations of \cite{Oka2}.
In comparison to \cite[Theorem 3.9]{Ouh5}, we note that the operators we consider here are in general no longer generators of sub-Markovian semigroups.
We will then apply the estimate to show that degenerate elliptic operators with smooth enough coefficients generate contraction holomorphic semigroups.

In order to formulate the main theorem, we need to introduce some notation.
Let $d \in \Ni$, $\Omega \subset \Ri^d$ be open with Lipschitz boundary and $\theta \in [0, \frac{\pi}{2})$. 
Let $c_{kl} \in W^{1,\infty}(\Omega, \Ci)$ for all $k, l \in \{1, \ldots, d\}$.
Define $C = (c_{kl})_{1 \leq k,l \leq d}$ and 
\begin{equation} \label{sector defi}
\Sigma_\theta = \{r \, e^{i \, \beta}: r \geq 0 \mbox{ and } |\beta| \leq \theta\}.
\end{equation}
Assume that 
\begin{equation} \label{values in sector}
(C(x) \, \xi, \xi) \in \Sigma_\theta
\end{equation}
for all $x \in \Omega$ and $\xi \in \Ci^d$.

Let $p \in (1, \infty)$.
Consider the operator $A_p$ in $L_p(\Omega)$ defined by
\[
A_p u = - \sum_{k,l=1}^d \D_l (c_{kl} \, \D_k u)
\]
on the domain
\[
D(A_p) = W^{2,p}(\Omega) \cap W^{1,p}_0(\Omega).
\]

If $p=2$ then
\begin{equation} \label{A2}
|\I (A_2 u, u)| \leq (\tan\theta) \, \R (A_2 u, u)
\end{equation}
for all $u \in D(A_2)$.
This follows immediately from integration by parts. 
If $p \neq 2$ the situation is quite different.
Write $C = R + i \, B$, where $R$ and $B$ are real matrices.
Let $R_a$ and $B_a$ be the anti-symmetric parts of $R$ and $B$ respectively, that is, $R_a = \frac{R - R^T}{2}$ and $B_a = \frac{B - B^T}{2}$.

The main result of this paper is as follows.

\begin{thrm} \label{main theorem}
Suppose $|1 - \frac{2}{p}| < \cos \theta$ and $B_a = 0$.
Then
\[
|\I (A_p u, |u|^{p-2} \, u \, \one_{[u \neq 0]})| \leq K \, \R (A_p u, |u|^{p-2} \, u \, \one_{[u \neq 0]})
\]
for all $u \in D(A_p)$, where 
\begin{equation} \label{K value}
K = \left\{
\begin{array}{ll}
\tan (\frac{\pi}{2} - \phi + \theta) & \mbox{if } R_a = 0, 
\\
\frac{\big( \frac{2}{\sin\phi} - 1 \big) \, \tan\theta + \cot\phi}{1 - (\tan \theta) \, \cot\phi} & \mbox{if } R_a \neq 0
\end{array}
\right.
\end{equation}
and $\phi = \arccos |1 - \frac{2}{p}|$.
\end{thrm}

Note that when the coefficient matrix $C$ consists of real entries and is symmetric, then one can choose $\theta = 0$ and \eqref{K value} gives 
\[
K = \tan (\frac{\pi}{2} - \phi) = \cot\phi = \frac{|p-2|}{2 \, \sqrt{p-1}},
\]
which is the constant obtained by Okazawa in \cite{Oka2}.

It is not difficult to see that $A_p$ is closable.
Let $\overline{A_p}$ be the closure of $A_p$.
Under the current conditions imposed on the coefficient matrix $C$ and the domain $\Omega$, we do not know whether $-\overline{A_p}$ is a generator of a $C_0$-semigroup.
If $\Omega = \Ri^d$ and $C$ consists of twice differentiable entries, then we prove the following generation result for $-\overline{A_p}$ based on Theorem \ref{main theorem}.

\begin{thrm} \label{holomorphic semigroup}
Let $\Omega = \Ri^d$.
Suppose $|1 - \frac{2}{p}| < \cos \theta$ and $B_a = 0$.
Suppose further that $c_{kl} \in W^{2,\infty}(\Ri^d)$ for all $k,l \in \{1, \ldots, d\}$.
Set $\phi = \arccos |1 - \frac{2}{p}|$.
Then the closure $-\overline{A_p}$ generates a holomorphic semigroup on $L_p(\Ri^d)$ with angle $\psi$ given by
\begin{equation} \label{psi}
\psi = \left\{
\begin{array}{ll}
\phi - \theta, & \mbox{if } R_a = 0,
\\
\frac{\pi}{2} - \arctan \Big( \frac{\big( \frac{2}{\sin\phi} - 1 \big) \, \tan\theta + \cot\phi}{1 - (\tan \theta) \, \cot\phi} \Big),
	& \mbox{if } R_a \neq 0.
\end{array}
\right.
\end{equation}
\end{thrm}

Note that 
\[
\psi_1 := \frac{\pi}{2} - \arctan \Big( \frac{\big( \frac{2}{\sin\phi} - 1 \big) \, \tan\theta + \cot\phi}{1 - (\tan \theta) \, \cot\phi} \Big)
\leq 
\phi - \theta
\]
since
\begin{equation} \label{angle comparison}
\tan \psi_1 
= \frac{1 - (\tan\theta) \, \cot\phi}{\big( \frac{2}{\sin\phi} - 1 \big) \, \tan\theta + \cot\phi}
\leq \frac{1 - (\tan \theta) \, \cot\phi}{\tan\theta + \cot\phi}
= \tan(\phi - \theta).
\end{equation}


It is also interesting that in the case when $R_a=0$, Theorem \ref{holomorphic semigroup} provides better angles of holomorphy compared with those of Stein's interpolations \cite[Proposition 3.12]{Ouh5} and \cite[Theorem 1]{Ste2}.
In the one-dimensional case, these better angles were also obtained in \cite[Corollary 1.3]{DE1}.
Other results about angles of holomorphy were considered in \cite[Theorem 1]{Wei}, \cite[Theorem 1.1]{Epp}, \cite[Theorem 1.4.2]{Dav2}, \cite[Theorem X.55]{RS2}, \cite{LiP1} and \cite[Theorems 3.12 and 3.13]{Ouh5}.

The holomorphic semigroup generated by $-\overline{A_p}$ in Theorem \ref{holomorphic semigroup} also possesses nice contractivity and consistency properties.

\begin{thrm} \label{contraction holomorphic semigroup}
Adopt the assumptions and notation as in Theorem \ref{holomorphic semigroup}.
Let $S^{(p)}$ be the semigroup generated by $-\overline{A_p}$ and $S$ the semigroup generated by $-\overline{A_2}$.
Then the following hold.
\begin{tabeleq}
\item \label{contraction 1} $S^{(p)}$ is contractive on $\Sigma_\gamma$, where
\begin{equation} \label{gamma}
\gamma 
= \left\{ 
\begin{array}{ll}
\psi & \mbox{if } R_a = 0,
\\
\psi \wedge \sup \, \big\{ \beta \in [0, \frac{\pi}{2}): (\tan\theta) \, \tan\beta < \frac{1}{3} \big\} & \mbox{if } R_a \neq 0.
\end{array}
\right.
\end{equation}
\item \label{consistence 1} $S^{(p)}$ is consistent with $S$ on $\Sigma_\psi$.
\end{tabeleq}
\end{thrm}

The outline of subsequent sections are as follows.
In Section 2 we provide some estimates on the coefficient matrix $C$.
These estimates are used to prove Theorem \ref{main theorem} in Section 3.
Theorems \ref{holomorphic semigroup} and \ref{contraction holomorphic semigroup} are proved in Section \ref{Rd domain}.

\section{Estimates on coefficients}

Let $\Omega$, $\theta$ and $C$ be as in Section \ref{Intro}.
In this section we provide some preliminary estimates on the coefficient matrix $C$ for later use.

Define 
\[
\R C = \frac{C + C^*}{2}
\quad
{\rm and}
\quad
\I C = \frac{C - C^*}{2i},
\]
where $C^*$ is the conjugate transpose of $C$.
Then $(\R C)(x)$ and $(\I C)(x)$ are self-adjoint for all $x \in \Omega$ and 
\begin{equation} \label{sa form}
C = \R C + i \, \I C.
\end{equation}
We will also decompose the coefficient matrix $C$ into
\begin{equation} \label{re form}
C = R + i \, B,
\end{equation}
where $R$ and $B$ are matrices with real entries.
Write $R = R_s + R_a$, where $R_s = \frac{R + R^T}{2}$ is the symmetric part of $R$ and $R_a = \frac{R - R^T}{2}$ is the anti-symmetric part of $R$.
Similarly $B = B_s + B_a$, where $B_s = \frac{B + B^T}{2}$ and $B_a = \frac{B - B^T}{2}$.
It follows from \eqref{sa form} and \eqref{re form} that
\[
\R C = R_s + i \, B_a
\quad
{\rm and}
\quad
\I C = B_s - i \, R_a.
\]

\begin{lemm} \label{Rs}
We have
\[
|(R_s \, \xi, \eta)| 
\leq \frac{1}{2} \, \Big( (R_s \, \xi, \xi) + (R_s \, \eta, \eta) \Big)
\]
for all $\xi, \eta \in \Ri^d$.
\end{lemm}

\begin{proof}
By hypothesis $C$ takes values in $\Sigma_\theta$.
This implies $((\R C) \, \xi, \xi) \geq 0$ for all $\xi \in \Ci^d$.
We deduce that $(R_s \, \xi, \xi) \geq 0$ for all $\xi \in \Ri^d$.
Finally we use polarisation to obtain the lemma.
\end{proof}

\begin{lemm} \label{W Bs < Rs}
We have
\[
|(B_s \, \xi, \eta)| 
\leq \frac{1}{2} \, (\tan \theta) \, \Big( (R_s \, \xi, \xi) + (R_s \, \eta, \eta) \Big)
\]
for all $\xi, \eta \in \Ri^d$.
\end{lemm}

\begin{proof}
Since $C$ takes values in $\Sigma_\theta$, we have
\begin{equation} \label{sector}
\big| \big( (\I C) \, \xi, \xi \big) \big| \leq (\tan \theta) \, \big( (\R C) \, \xi, \xi \big)
\end{equation}
for all $\xi \in \Ci^d$.
It follows that
\[
|(B_s \, \xi, \xi)| \leq (\tan \theta) \, (R_s \, \xi, \xi)
\]
for all $\xi \in \Ri^d$.
Finally we use polarisation to obtain
\[
|(B_s \, \xi, \eta)| 
\leq (\tan \theta) \, (R_s \, \xi, \xi)^{1/2} \, (R_s \, \eta, \eta)^{1/2}
\leq \frac{1}{2} \, (\tan \theta) \, \Big( (R_s \, \xi, \xi) + (R_s \, \eta, \eta) \Big)
\]
for all $\xi, \eta \in \Ri^d$ as required.
\end{proof}

\begin{lemm} \label{sector equiv form}
We have
\[
\big| (B_s \, \xi, \xi) + (B_s \, \eta, \eta) - 2 \, (R_a \, \xi, \eta) \big|
\leq 
(\tan\theta) \, \Big( (R_s \, \xi, \xi) + (R_s \, \eta, \eta) + 2 \, (B_a \, \xi, \eta) \Big)
\]
for all $\xi, \eta \in \Ri^d$.
\end{lemm}

\begin{proof}
Let $\xi, \eta \in \Ri^d$.
Then
\[
\big( (\I C) \, (\xi + i \, \eta), \xi + i \, \eta \big) 
= (B_s \, \xi, \xi) + (B_s \, \eta, \eta) - 2 \, (R_a \, \xi, \eta)
\]
and 
\[
\big( (\R C) \, (\xi + i \, \eta), \xi + i \, \eta \big) 
= (R_s \, \xi, \xi) + (R_s \, \eta, \eta) + 2 \, (B_a \, \xi, \eta).
\]
The claim is now immediate from $\eqref{sector}$.
\end{proof}

\begin{lemm} \label{Ra < Rs}
Suppose $B_a = 0$.
Then 
\[
\big| (R_a \, \xi, \eta) \big| 
\leq (\tan \theta) \, \Big( (R_s \, \xi, \xi) + (R_s \, \eta, \eta) \Big)
\]
for all $\xi, \eta \in \Ri^d$.
\end{lemm}

\begin{proof}
Since $B_a = 0$, Lemma \ref{sector equiv form} gives
\[
\big| (B_s \, \xi, \xi) + (B_s \, \eta, \eta) - 2 \, (R_a \, \xi, \eta) \big|
\leq 
(\tan\theta) \, \Big( (R_s \, \xi, \xi) + (R_s \, \eta, \eta) \Big).
\]
The result now follows from the triangle inequality and Lemma \ref{W Bs < Rs}.
\end{proof}

\begin{lemm} \label{W R<trR}
Let $Q$ be a positive matrix and $U$ a complex $d \times d$ matrix.
Then
\[
(Q \, U \, \xi, U \, \xi) \leq \tr(U^* \, Q \, U) \, \|\xi\|^2
\]
for all $\xi \in \Ci^d$.
\end{lemm}

\begin{proof}
Since $Q$ is a positive matrix, we have $(Q \, U \, \xi, U \, \xi) \geq 0$ for all $\xi \in \Ci^d$.
It follows that $U^* \, Q \, U \geq 0$.
Hence $U^* \, Q \, U \leq \tr(U^* \, Q \, U) \, I$, where $I$ denotes the identity matrix.
This justifies the claim.
\end{proof}

\begin{lemm} \label{positive definite}
We have the following.
\begin{tabel}
\item $(R_s \, \xi, \xi) \geq 0$ for all $\xi \in \Ci^d$. 
\label{Rs positive 1}
\item $\left( \big( (\tan\theta) \, R_s \pm B_s \big) \, \xi, \xi \right) \geq 0$ for all $\xi \in \Ci^d$.
\label{Rs positive 2}
\item Suppose $B_a = 0$. 
Then $\left( \big( 2 \, (\tan\theta) \, R_s \pm i \, R_a \big) \, \xi, \xi \right) \geq 0$ for all $\xi \in \Ci^d$.
\label{Rs positive 3}
\end{tabel}
\end{lemm}

\begin{proof}
Let $\xi \in \Ci^d$.
Write $\xi = \xi_1 + i \, \xi_2$, where $\xi_1, \xi_2 \in \Ri^d$.
We note that 
\[
(R_s \, \xi, \xi) = (R_s \, \xi_1, \xi_1) + (R_s \, \xi_2, \xi_2)
\]
and
\[
(B_s \, \xi, \xi) = (B_s \, \xi_1, \xi_1) + (B_s \, \xi_2, \xi_2).
\]
Also
\[
(R_a \, \xi, \xi) = -2i \, (R_a \, \xi_1, \xi_2).
\]
The claim now follows from Lemmas \ref{Rs}, \ref{W Bs < Rs} and \ref{Ra < Rs}.
\end{proof}

Next let $\alpha \in (-\frac{\pi}{2} + \theta, \frac{\pi}{2} - \theta)$ and write $C_\alpha = e^{i \alpha} \, C$.
In a similar manner as above, we define $\R (C_\alpha)$, $\I (C_\alpha)$, $R_\alpha$, $B_\alpha$, $R_{s,\alpha}$, $R_{a,\alpha}$, $B_{s,\alpha}$ and $B_{a,\alpha}$.
Note that we also have
\[
\R (C_\alpha) = R_{s,\alpha} + i \, B_{a,\alpha}
\quad
{\rm and}
\quad
\I (C_\alpha) = B_{s,\alpha} - i \, R_{a,\alpha}.
\]

\begin{lemm} \label{W Oleinik 2}
Let $j \in \{1, \ldots, d\}$.
Suppose $U$ is a complex $d \times d$ matrix with $U^T = U$.
Then 
\[
|\tr ((\D_j C_\alpha) \, U)|^2 \leq M \, \tr(U \, R_{s,\alpha} \, \overline{U}),
\]
where 
\[
M = 32 \, d \, \big( 1 + \tan(\theta + \alpha) \big)^2 \, \|\D_l^2 C\|_\infty.
\]
\end{lemm}

\begin{proof}
It follows from \cite[Corollary 2.6]{Do1} that 
\begin{eqnarray*}
|\tr ((\D_j C_\alpha) \, U)|^2 
&\leq& 32 \, d \, \big( 1 + \tan(\theta + \alpha) \big)^2 \, \|\D_l^2 (e^{i \alpha} \, C)\|_\infty \, 
		\tr(U \, R_{s,\alpha} \, \overline{U})
\\
&\leq& 32 \, d \, \big( 1 + \tan(\theta + \alpha) \big)^2 \, \|\D_l^2 C\|_\infty \, 
		\tr(U \, R_{s,\alpha} \, \overline{U})
\end{eqnarray*}
as required.
\end{proof}

\begin{lemm} \label{C alpha identities}
Suppose $B_a = 0$.
Then the following hold.
\begin{tabeleq}
\item \label{Re alpha} $\R (C_\alpha) = R_s \, \cos\alpha - B_s \, \sin\alpha + i \, R_a \, \sin\alpha$.
\item \label{Im alpha} $\I (C_\alpha) = R_s \, \sin\alpha + B_s \, \cos\alpha - i \, R_a \, \cos\alpha$.
\item \label{R alpha} 
$R_\alpha = R_s \, \cos\alpha + R_a \, \cos\alpha - B_s \, \sin\alpha$,

$R_{s,\alpha} = R_s \, \cos\alpha - B_s \, \sin\alpha$,

$R_{a,\alpha} = R_a \, \cos\alpha$.
\item \label{B alpha}
$B_\alpha = R_s \, \sin\alpha + R_a \, \sin\alpha + B_s \, \cos\alpha$,

$B_{s,\alpha} = R_s \, \sin\alpha + B_s \, \cos\alpha$,

$B_{a,\alpha} = R_a \, \sin\alpha$.
\end{tabeleq} 
\end{lemm}

\begin{proof}
These identities follow directly from the definition of $C$ and $C_\alpha$.
\end{proof}

\section{Sectorial property}

Let $p \in (1, \infty)$.
Let $\Omega$, $\theta$, $C$ and $A_p$ be as in Section \ref{Intro}.
In this section we prove Theorem \ref{main theorem}.
A convenient tool that we will use repeatedly is the formula of integration by parts in Sobolev spaces given in the next theorem.
The theorem is immediate from the proof of \cite[Proposition 3.5]{MS}.
We emphasise that we do not require $C = C^T$ in this theorem (cf.\ \cite[Theorem 3.1]{MS} for the same result but with extra assumption that $C = C^T$).

\begin{thrm} \label{BG integration by parts on Bp}
Let $u \in D(A_p)$.
Then
\begin{eqnarray}
\int_{[u \neq 0]} (A_p u) \, |u|^{p-2} \, \overline{u}
& = & \int_{[u \neq 0]} |u|^{p-2} \, (C \, \nabla \overline{u}, \nabla \overline{u})
\nonumber
\\
& &
{} + (p - 2) \, \int_{[u \neq 0]} |u|^{p-4} \, 
	\big( C \, \R(u \, \nabla \overline{u}), \R(u \, \nabla \overline{u}) \big)
\nonumber
\\
& &
	{} - i \, (p - 2) \, \int_{[u \neq 0]} |u|^{p-4} \, 
	\big( C \, \R(u \, \nabla \overline{u}), \I(u \, \nabla \overline{u}) \big).
	\label{ibpMS}
\end{eqnarray}
\end{thrm}

Using Theorem \ref{BG integration by parts on Bp} we obtain the following.

\begin{prop} \label{Re and Im}
Let $u \in D(A_p)$.
Write $u \, \nabla \overline{u} = \xi + i\, \eta$, where $\xi, \eta \in \Ri^d$.
Then
\begin{eqnarray*}
\R (A_p u, |u|^{p-2} \, u \, \one_{[u \neq 0]})
&=& \int_{[u \neq 0]} |u|^{p-4} \, \Big( 
	(p-1) \, (R_s \, \xi, \xi) + (R_s \, \eta, \eta) \Big.
	\\
	&&
	\Big. \hskip 23mm + (p-2) \, (B_s \, \xi, \eta) + p \, (B_a \, \xi, \eta) \Big)
\end{eqnarray*}
and
\begin{eqnarray*}
\I (A_p u, |u|^{p-2} \, u \, \one_{[u \neq 0]})
&=& \int_{[u \neq 0]} |u|^{p-4} \, \Big(
	(p-1) \, (B_s \, \xi, \xi) + (B_s \, \eta, \eta) \Big.
	\\
	&&
	\Big. \hskip 23mm - (p-2) \, (R_s \, \xi, \eta) - p \, (R_a \, \xi, \eta) \Big).
\end{eqnarray*}
\end{prop}

\begin{proof}
We will prove the first inequality only.
The second is similar.

Consider \eqref{ibpMS}.
We have
\begin{eqnarray*}
|u|^2 \, (C \, \nabla \overline{u}, \nabla \overline{u})
& = & (C \, u \, \nabla \overline{u}, u \, \nabla \overline{u})
	= \big( C (\xi + i \, \eta), \xi + i \, \eta \big)
\\
& = & (R \, \xi, \xi) + (R \, \eta, \eta) + (B \, \xi, \eta) - (B \, \eta, \xi)
	\\
	& {} & {} - i \, \big( (R \, \eta, \xi) - (R \, \xi, \eta) + (B \, \xi, \xi) + (B \, \eta, \eta) \big).
\end{eqnarray*}
Therefore
\begin{eqnarray*}
\R \big( |u|^{2} \, (C \, \nabla \overline{u}, \nabla \overline{u}) \big)
& = & (R \, \xi, \xi) + (R \, \eta, \eta) + (B \, \xi, \eta) - (B \, \eta, \xi)
\\
& = & (R_s \, \xi, \xi) + (R_s \, \eta, \eta) + 2 \, (B_a \, \xi, \eta).
\end{eqnarray*}
Also
\[
\R \big( C \, \R(u \, \nabla \overline{u}), \R(u \, \nabla \overline{u}) \big)
= \R (C \, \xi, \xi)
= (R \, \xi, \xi)
= (R_s \, \xi , \xi).
\]
Similarly
\[
\R \big( i \, \big( C \, \R(u \, \nabla \overline{u}), \I(u \, \nabla \overline{u}) \big) \big)
= \R \big( i \, (C \, \xi, \eta) \big)
= -(B \, \xi, \eta)
= -(B_s \, \xi , \eta) -(B_a \, \xi , \eta).
\]
Hence taking the real parts on both sides of \eqref{ibpMS} yields the result.
\end{proof}

%
%

The following lemma is essential in the proof of Theorem \ref{main theorem}.

\begin{lemm} \label{ps}
Suppose $|1 - \frac{2}{p}| < \cos \theta$.
Let $\phi = \arccos |1 - \frac{2}{p}|$.
Then
\begin{eqnarray*}
&& \big( \tan (\frac{\pi}{2} - \phi) + \tan\theta \big) \, 
	\big( (R_s \, \xi, \xi) + (R_s \, \eta, \eta) \big)
\\
&\leq& 
\tan (\frac{\pi}{2} - \phi + \theta) \, 
	\big( (R_s \, \xi, \xi) + (R_s \, \eta, \eta) + \frac{p-2}{\sqrt{p-1}} \, (B_s \, \xi, \eta) \big)
\end{eqnarray*}
for all $\xi, \eta \in \Ri^d$.
\end{lemm}

\begin{proof}
First note that 
\begin{eqnarray}
&& \tan(\frac{\pi}{2} - \phi) \, (\tan\theta) \, \big( (R_s \, \xi, \xi) + (R_s \, \eta, \eta) \big)
	+ \frac{p-2}{\sqrt{p-1}} \, (B_s \, \xi', \eta)
\nonumber
\\
&\geq& 
\tan(\frac{\pi}{2} - \phi) \, (\tan\theta) \, \big( (R_s \, \xi, \xi) + (R_s \, \eta, \eta) \big)
	- \frac{|p-2|}{\sqrt{p-1}} \, |(B_s \, \xi, \eta)|
\nonumber
\\
&=& \tan(\frac{\pi}{2} - \phi) \, \Big( (\tan\theta) \, \big( (R_s \, \xi, \xi) + (R_s \, \eta, \eta) \big)
	- 2 \, |(B_s \, \xi, \eta)| \Big) \geq 0
\label{positive term}
\end{eqnarray}
as $\tan (\frac{\pi}{2} - \phi) = \cot(\phi) = \frac{|p-2|}{2 \, \sqrt{p-1}}$ and we used Lemma \ref{W Bs < Rs} in the last step.
We also deduce from the hypotheses that $\tan (\frac{\pi}{2} - \phi + \theta) \geq 0$.
Therefore
\begin{eqnarray*}
&& \big( \tan (\frac{\pi}{2} - \phi) + \tan\theta \big) \, 
	\big( (R_s \, \xi, \xi) + (R_s \, \eta, \eta) \big)
\\
&\leq& \big( \tan(\frac{\pi}{2} - \phi) + \tan\theta \big) \, \big( (R_s \, \xi, \xi) + (R_s \, \eta, \eta) \big)
\\
&& {} + \tan (\frac{\pi}{2} - \phi + \theta) \, 
	\Big( \tan(\frac{\pi}{2} - \phi) \, (\tan\theta) \, \big( (R_s \, \xi, \xi) + (R_s \, \eta, \eta) \big)
	+ \frac{p-2}{\sqrt{p-1}} \, (B_s \, \xi, \eta) \Big)
\\
&=& \tan (\frac{\pi}{2} - \phi + \theta) \, 
	\big( (R_s \, \xi, \xi) + (R_s \, \eta, \eta) + \frac{p-2}{\sqrt{p-1}} \, (B_s \, \xi, \eta) \big),
\end{eqnarray*}
where we used \eqref{positive term} in the first step.
\end{proof}

Next we prove Theorem \ref{main theorem}.

\begin{proof}[{\bf Proof of Theorem \ref{main theorem}}]
Let $u \in D(A_p)$.
Write $u \, \nabla \overline{u} = \xi + i \, \eta$, where $\xi, \eta \in \Ri^d$.
By Proposition \ref{Re and Im}, it suffices to show that 
\begin{eqnarray}
&& \big| (p-1) \, (B_s \, \xi, \xi) + (B_s \, \eta, \eta) 
	- (p-2) \, (R_s \, \xi, \eta) - p \, (R_a \, \xi, \eta) \big|
\nonumber
\\
&\leq& 
K \, \big( (p-1) \, (R_s \, \xi, \xi) + (R_s \, \eta, \eta) + (p-2) \, (B_s \, \xi, \eta) \big),
\label{main ineq 1}
\end{eqnarray}
where $K$ is defined by \eqref{K value}.
Set $\xi' = \sqrt{p-1} \, \xi$.
Then $\eqref{main ineq 1}$ is equivalent to
\begin{eqnarray}
&& \big| (B_s \, \xi', \xi') + (B_s \, \eta, \eta) 
	- \frac{p-2}{\sqrt{p-1}} \, (R_s \, \xi', \eta) - \frac{p}{\sqrt{p-1}} \, (R_a \, \xi', \eta) \big|
\nonumber
\\
&\leq& 
K \, \Big( (R_s \, \xi', \xi') + (R_s \, \eta, \eta) + \frac{p-2}{\sqrt{p-1}} \, (B_s \, \xi', \eta) \Big).
\label{main ineq 1 equiv form}
\end{eqnarray}
Note that by Lemma \ref{Rs} we have
\begin{equation} \label{Rs estimate}
\frac{|p-2|}{\sqrt{p-1}} \, \big| (R_s \, \xi', \eta) \big|
\leq \tan (\frac{\pi}{2} - \phi) \, \big( (R_s \, \xi', \xi') + (R_s \, \eta, \eta) \big)
\end{equation}
as $\tan (\frac{\pi}{2} - \phi) = \cot(\phi) = \frac{|p-2|}{2 \, \sqrt{p-1}}$.

Now we consider two cases.
\\
{\bf Case 1:} Suppose $R_a = 0$. 
Using Lemma \ref{W Bs < Rs} again we obtain
\begin{equation} \label{Bs estimate}
\big| (B_s \, \xi', \xi') + (B_s \, \eta, \eta) \big| 
\leq (\tan\theta) \, \big( (R_s \, \xi', \xi') + (R_s \, \eta, \eta) \big).
\end{equation}
It follows that 
\begin{eqnarray*} 
&& \big| (B_s \, \xi', \xi') + (B_s \, \eta, \eta) 
	- \frac{p-2}{\sqrt{p-1}} \, (R_s \, \xi', \eta)  - \frac{p}{\sqrt{p-1}} \, (R_a \xi', \eta) \big|
\\
&=& \big| (B_s \, \xi', \xi') + (B_s \, \eta, \eta) - \frac{p-2}{\sqrt{p-1}} \, (R_s \, \xi', \eta) \big|
\\
&\leq& \big( \tan (\frac{\pi}{2} - \phi) + \tan\theta \big) \, 
	\big( (R_s \, \xi', \xi') + (R_s \, \eta, \eta) \big)
\\
&\leq& \tan (\frac{\pi}{2} - \phi + \theta) \, 
	\big( (R_s \, \xi', \xi') + (R_s \, \eta, \eta) + \frac{p-2}{\sqrt{p-1}} \, (B_s \, \xi', \eta) \big),
\end{eqnarray*}
where we used $R_a = 0$ in the first step, \eqref{Bs estimate} and \eqref{Rs estimate} in the second step and Lemma \ref{ps} in the last step.

Hence \eqref{main ineq 1 equiv form} is valid and the result follows in this case.
\\
{\bf Case 2:} Suppose $R_a \neq 0$.
We rewrite the left hand side of \eqref{main ineq 1 equiv form} as
\begin{eqnarray*}
L &:=& \Big| \Big( (B_s \, \xi', \xi') + (B_s \, \eta, \eta) - 2 \, (R_a \, \xi', \eta) \Big)
	- \frac{p-2}{\sqrt{p-1}} \, (R_s \, \xi', \eta) 
	\\
	&& {} - \big( \frac{p}{\sqrt{p-1}} - 2 \big) \, (R_a \, \xi', \eta) \Big|.
\end{eqnarray*}
(Note that $\frac{p}{\sqrt{p-1}} \geq 2$ for all $p \in (1, \infty)$.)
Since $B_a = 0$, it follows from Lemma \ref{sector equiv form} that
\begin{equation} \label{Bs estimate another}
\big| (B_s \, \xi', \xi') + (B_s \, \eta, \eta) - 2 \, (R_a \, \xi', \eta) \big|
\leq (\tan\theta) \, \big( (R_s \, \xi', \xi') + (R_s \, \eta, \eta) \big).
\end{equation}
Next we deduce from Lemma \ref{Ra < Rs} that
\begin{equation} \label{Ra estimate}
\big( \frac{p}{\sqrt{p-1}} - 2 \big) \, \big| (R_a \, \xi', \eta) \big|
\leq \big( \frac{2}{\sin\phi} - 2 \big) \, (\tan\theta) \, \big( (R_s \, \xi', \xi') + (R_s \, \eta, \eta) \big)
\end{equation}
as $\sin\phi = \frac{2 \, \sqrt{p-1}}{p}$.
Now it follows from \eqref{Rs estimate}, \eqref{Bs estimate another} and \eqref{Ra estimate} that
\begin{eqnarray*}
L
&\leq& \Big( \big( \frac{2}{\sin\phi} - 1 \big) \, \tan\theta + \tan(\frac{\pi}{2} - \phi) \Big) \,
		\big( (R_s \, \xi', \xi') + (R_s \, \eta, \eta) \big)
\\
&=& \frac{\big( \frac{2}{\sin\phi} - 1 \big) \, \tan\theta + \tan(\frac{\pi}{2} - \phi)}{\tan\theta + \tan (\frac{\pi}{2} - \phi)} \, 
	\big( \tan\theta + \tan (\frac{\pi}{2} - \phi) \big) \, 
	\big( (R_s \, \xi', \xi') + (R_s \, \eta, \eta) \big)
\\
&\leq& \frac{\big( \frac{2}{\sin\phi} - 1 \big) \, \tan\theta + \tan(\frac{\pi}{2} - \phi)}{\tan\theta + \tan (\frac{\pi}{2} - \phi)} \, 
	\tan (\frac{\pi}{2} - \phi + \theta) \, 
	\Big( (R_s \, \xi', \xi') + (R_s \, \eta, \eta) 
	\\
	&& {} + \frac{p-2}{\sqrt{p-1}} \, (B_s \, \xi', \eta) \Big)
\\
&=& \frac{\big( \frac{2}{\sin\phi} - 1 \big) \, \tan\theta + \tan(\frac{\pi}{2} - \phi)}{1 - (\tan \theta) \, \tan (\frac{\pi}{2}-\phi)} \, 
	\Big( (R_s \, \xi', \xi') + (R_s \, \eta, \eta) 
	+ \frac{p-2}{\sqrt{p-1}} \, (B_s \, \xi', \eta) \Big),
\end{eqnarray*}
where we used Lemma \ref{ps} in the second step.

Hence \eqref{main ineq 1 equiv form} is also valid in this case.
\end{proof}

\section{Generation of contraction holomorphic semigroup} \label{Rd domain}

Let $\Omega = \Ri^d$ and $\theta \in [0, \frac{\pi}{2})$.
We assume $c_{kl} \in W^{2,\infty}(\Ri^d, \Ci)$ for all $k, l \in \{1, \ldots, d\}$.
Assume further that $(C(x) \, \xi, \xi) \in \Sigma_\theta$ for all $x \in \Ri^d$ and $\xi \in \Ci^d$, where $C = (c_{kl})_{1 \leq k,l \leq d}$ and $\Sigma_\theta$ is defined by \eqref{sector defi}.

Let $p \in (1, \infty)$.
We will prove in Proposition \ref{Ap closable} that $A_p$ is closable.
Let $\overline{A_p}$ be the closure of $A_p$.
We will show in this section that $-\overline{A_p}$ generates a holomorphic semigroup on $L_p(\Ri^d)$ which is contractive on a sector.
This is the content of Theorems \ref{holomorphic semigroup} and \ref{contraction holomorphic semigroup}.

First we introduce some more definitions.
Let $q$ be such that $\frac{1}{p} + \frac{1}{q} = 1$.
Define 
\begin{equation} \label{Hq}
H_q u = - \sum_{k,l=1}^d \D_k (\overline{c_{kl}} \, \D_l u)
\end{equation}
on the domain
\[
D(H_q) = C_c^\infty(\Ri^d).
\]
Define 
\[
B_p = (H_q)^*,
\]
which is the dual of $H_q$.
Then $B_p$ is closed by \cite[Theorem III.5.29]{Kat1}.
Also note that $W^{2,p}(\Ri^d) \subset D(B_p)$ and
\[
B_p u = - \sum_{k,l=1}^d \D_l (c_{kl} \, \D_k u)
\]
for all $u \in W^{2,p}(\Ri^d)$.

\begin{prop} \label{Ap closable}
The operator $A_p$ is closable.
\end{prop}

\begin{proof}
Since $A_p \subset B_p$ and $B_p$ is closed, the operator $A_p$ is closable.
\end{proof}

It turns out that $B_p = \overline{A_p}$ under certain conditions, as shown in the following proposition.

\begin{prop} \label{Ap m accretive}
Suppose $|1 - \frac{2}{p}| \leq \cos\theta$ and $B_a = 0$.
Then $\overline{A_p} = B_p$.
Moreover, $\overline{A_p}$ is $m$-accretive.
\end{prop}

\begin{proof}
By \cite[Proposition 4.9]{Do1} the operator $B_p$ is $m$-accretive and the space $C_c^\infty(\Ri^d)$ of test functions is a core for $B_p$.
It follows that $\overline{A_p} = B_p$ and $A_p$ is $m$-accretive as claimed.
\end{proof}

Using Theorem \ref{main theorem} we are now able to prove the generation result in Theorem \ref{holomorphic semigroup}.

\begin{proof}[{\bf Proof of Theorem \ref{holomorphic semigroup}}]
It follows from Theorem \ref{main theorem} that 
\[
|\I (\overline{A_p} u, |u|^{p-2} \, u \, \one_{[u \neq 0]})| 
\leq K \, \R (\overline{A_p} u, |u|^{p-2} \, u \, \one_{[u \neq 0]})
\]
for all $u \in D(\overline{A_p})$, where $K$ is defined by \eqref{K value}.
Therefore the interior $\Sigma_{\pi - \arctan(K)}^\circ \subset \rho(-\overline{A_p})$ by \cite[Theorem 1.3.9]{Paz} and Proposition \ref{Ap m accretive}, where $\rho(-\overline{A_p})$ denotes the resolvent set of $-\overline{A_p}$.
Moreover,
\begin{equation} \label{resovent bound}
\|(\lambda + \overline{A_p})^{-1}\|_{p \to p} \leq \frac{1}{\dist(\lambda,S(-\overline{A_p}))}
\end{equation}
for all $\lambda \in \Sigma_{\pi - \arctan(K)}^\circ$, where $S(-\overline{A_p})$ is the numerical range of $-\overline{A_p}$ defined by
\[
S(-\overline{A_p}) 
= \big\{ - \big( \overline{A_p} u, |u|^{p-2} \, u \, \one_{[u \neq 0]} \big): 
		u \in D(\overline{A_p}) \mbox{ and } \|u\|_p = 1 \big\}.
\]

Let $\varepsilon \in (0, \pi - \arctan(K))$. 
Then $\dist(\lambda,S(-\overline{A_p})) \geq (\sin\varepsilon) \, |\lambda|$ for all $\lambda \in \Sigma_{\pi - \arctan(K) - \varepsilon}$.
Therefore \eqref{resovent bound} implies 
\[
\|(\lambda + \overline{A_p})^{-1}\|_{p \to p} \leq \frac{1}{(\sin\varepsilon) \, |\lambda|}
\]
for all $\lambda \in \Sigma_{\pi - \arctan(K) - \varepsilon}$.
Hence we deduce from \cite[Theorem 2.5.2(c)]{Paz} that $-\overline{A_p}$ generates a holomorphic semigroup on $L_p(\Ri^d)$ with angle $\psi = \frac{\pi}{2} - \arctan(K)$.
\end{proof}

Our next aim is to show Theorem \ref{contraction holomorphic semigroup}.
We will do this by first showing that $-B_p$ generates a holomorphic semigroup which is contractive on a sector.
This together with Proposition \ref{Ap m accretive} imply the theorem.
We first obtain some preliminary results.

In what follows we let $B_{p,\alpha} = e^{i \alpha} B_p$ for all $\alpha \in (-\frac{\pi}{2} + \theta, \frac{\pi}{2}  - \theta)$ and adopt the notation used in Lemmas \ref{W Oleinik 2} and \ref{C alpha identities}.
We aim to show that $B_{p,\alpha}$ is an $m$-accretive operator for all $\alpha$ in a suitable range.
Following \cite{WongDzung} and \cite{Do1} we need two crucial inequalities for $B_{p,\alpha}$ in order to do this.
The first inequality is given by the next proposition.
The second inequality is derived in Proposition \ref{W 2nd ineq}.

\begin{prop} \label{W 1st ineq}
Suppose $B_a = 0$.
Let $p \in (1, \infty)$ be such that $|1 - \frac{2}{p}| < \cos \theta$.
Let $\alpha \in (-\psi, \psi)$, where $\psi$ is given by \eqref{psi}.
Then
\[
\R (B_{p,\alpha} u, |u|^{p-2} \, u \, \one_{[u \neq 0]}) \geq 0
\]
for all $u \in W^{2,p}(\Ri^d)$.
\end{prop}

\begin{proof}
Let $u \in W^{2,p}(\Ri^d)$.
It follows from Theorem \ref{BG integration by parts on Bp} that
\begin{eqnarray} 
(B_{p,\alpha} u, |u|^{p-2} \, u \, \one_{[u \neq 0]})
& = & \int_{[u \neq 0]} |u|^{p-2} \, (C_\alpha \, \nabla \overline{u}, \nabla \overline{u})
\nonumber
\\
& &
{} + (p - 2) \, \int_{[u \neq 0]} |u|^{p-4} \, 
	\big( C_\alpha \, \R(u \, \nabla \overline{u}), \R(u \, \nabla \overline{u}) \big)
\nonumber
\\
& &
	{} - i \, (p - 2) \, \int_{[u \neq 0]} |u|^{p-4} \, 
	\big( C_\alpha \, \R(u \, \nabla \overline{u}), \I(u \, \nabla \overline{u}) \big).
\label{W MS expansion}
\end{eqnarray}
Write $u \, \nabla \overline{u} = \xi + i \, \eta$, where $\xi, \eta \in \Ri^d$.
Then
\begin{eqnarray*}
|u|^2 \, (C_\alpha \, \nabla \overline{u}, \nabla \overline{u})
& = & (C_\alpha \, u \, \nabla \overline{u}, u \, \nabla \overline{u})
	= \big( C_\alpha (\xi + i \, \eta), \xi + i \, \eta \big)
\\
& = & (R_\alpha \, \xi, \xi) + (R_\alpha \, \eta, \eta) + (B_\alpha \, \xi, \eta) - (B_\alpha \, \eta, \xi)
	\\
	& {} & {} + i \, \big( (R_\alpha \, \eta, \xi) - (R_\alpha \, \xi, \eta) + (B_\alpha \, \xi, \xi) + (B_\alpha \, \eta, \eta) \big).
\end{eqnarray*}
Therefore
\begin{eqnarray*}
\R \big( |u|^{2} \, (C_\alpha \, \nabla \overline{u}, \nabla \overline{u}) \big)
& = & (R_\alpha \, \xi, \xi) + (R_\alpha \, \eta, \eta) + (B_\alpha \, \xi, \eta) - (B_\alpha \, \eta, \xi)
\\
& = & (R_{s,\alpha} \, \xi, \xi) + (R_{s,\alpha} \, \eta, \eta) + 2 \, (B_{a,\alpha} \, \xi, \eta).
\end{eqnarray*}
We also have
\[
\R \big( C_\alpha \, \R(u \, \nabla \overline{u}), \R(u \, \nabla \overline{u}) \big)
= \R (C_\alpha \, \xi, \xi)
= (R_\alpha \, \xi, \xi)
= (R_{s,\alpha} \, \xi , \xi).
\]
Similarly
\[
\R \big( i \, \big( C_\alpha \, \R(u \, \nabla \overline{u}), \I(u \, \nabla \overline{u}) \big) \big)
= \R \big( i \, (C_\alpha \, \xi, \eta) \big)
= -(B_\alpha \, \xi, \eta)
= -(B_{s,\alpha} \, \xi , \eta) - (B_{a,\alpha} \, \xi , \eta).
\]
Hence taking the real parts on both sides of \eqref{W MS expansion} yields
\begin{eqnarray}
&& \R (B_{p,\alpha} u, |u|^{p-2} \, u \, \one_{[u \neq 0]})
\nonumber
\\
&=& \int_{[u \neq 0]} |u|^{p-4} \, \Big( (p-1) \, (R_{s,\alpha} \, \xi, \xi) + (R_{s,\alpha} \, \eta, \eta) 
	+ p \, (B_{a,\alpha} \, \xi, \eta)	
	+ (p-2) \, (B_{s,\alpha} \, \xi, \eta) \Big)
\nonumber
\\
& = & \int_{[u \neq 0]} |u|^{p-4} \, \Big( (R_{s,\alpha} \, \xi', \xi') + (R_{s,\alpha} \, \eta, \eta) 
	+ \frac{p}{\sqrt{p-1}} \, (B_{a,\alpha} \, \xi', \eta)
	+ \frac{p-2}{\sqrt{p-1}} \, (B_{s,\alpha} \, \xi', \eta) \Big),
\nonumber
\\
&& \label{int P}
\end{eqnarray}
where $\xi' = \sqrt{p-1} \, \xi$.
Set 
\begin{equation} \label{P}
P 
= (R_{s,\alpha} \, \xi', \xi') + (R_{s,\alpha} \, \eta, \eta) 
	+ \frac{p}{\sqrt{p-1}} \, (B_{a,\alpha} \, \xi', \eta)
	+ \frac{p-2}{\sqrt{p-1}} \, (B_{s,\alpha} \, \xi', \eta).
\end{equation}
We will show that $P \geq 0$.
We consider 2 cases.
\\
{\bf Case 1:} Suppose $R_a = 0$.
Note that $\cot\phi = \frac{|p-2|}{2 \, \sqrt{p-1}}$.
We have
\begin{equation} \label{case 1 estimate 1}
\big| (\sin\alpha) \, \big( (B_s \, \xi', \xi') + (B_s \, \eta, \eta) \big) \big|
\leq 
\sin(|\alpha|) \, (\tan\theta) \, \big( (R_s \, \xi', \xi') + (R_s \, \eta, \eta) \big)
\end{equation}
and 
\begin{equation} \label{case 1 estimate 2}
\Big| \frac{p-2}{\sqrt{p-1}} \, (\cos\alpha) \, (B_s \, \xi', \eta) \Big|
\leq 
(\cot\phi) \, (\cos\alpha) \, (\tan\theta) \, \big( (R_s \, \xi', \xi') + (R_s \, \eta, \eta) \big).
\end{equation}
by Lemma \ref{W Bs < Rs}.
Also
\begin{equation} \label{case 1 estimate 3}
\Big| \frac{p-2}{\sqrt{p-1}} \, (\sin\alpha) \, (R_s \, \xi', \eta) \Big|
\leq 
(\cot\phi) \, \sin(|\alpha|) \, \big( (R_s \, \xi', \xi') + (R_s \, \eta, \eta) \big)
\end{equation}
by Lemma \ref{Rs}.
Since $R_a = 0$, Lemma \ref{C alpha identities}\ref{B alpha} gives $B_{a,\alpha} = (\sin\alpha) \, R_a = 0$.
It follows from Lemma \ref{C alpha identities}, \eqref{P}, \eqref{case 1 estimate 1}, \eqref{case 1 estimate 2} and \eqref{case 1 estimate 3} that
\begin{eqnarray*}
P 
&=& (R_{s,\alpha} \, \xi', \xi') + (R_{s,\alpha} \, \eta, \eta) 
	+ \frac{p-2}{\sqrt{p-1}} \, (B_{s,\alpha} \, \xi', \eta)
\\
&=& (\cos\alpha) \, \big( (R_s \, \xi', \xi') + (R_s \, \eta, \eta) \big)
	- (\sin\alpha) \, \big( (B_s \, \xi', \xi') + (B_s \, \eta, \eta) \big)
	\\
	&& {} 	+ \frac{p-2}{\sqrt{p-1}} \, (\sin\alpha) \, (R_s \, \xi', \eta) 
			+ \frac{p-2}{\sqrt{p-1}} \, (\cos\alpha) \, (B_s \, \xi', \eta)
\\
&\geq& \Big( \cos\alpha - \sin(|\alpha|) \, \tan\theta 
		- (\cot\phi) \, \sin(|\alpha|) - (\cot\phi) \, (\cos\alpha) \, \tan\theta
		\Big) \, \big( (R_s \, \xi', \xi') + (R_s \, \eta, \eta) \big)
\\
&\geq& 0,
\end{eqnarray*}
where we used the fact that $\alpha \in (-\psi,\psi)$ in the last step.
Hence we deduce from \eqref{int P} that $\R (B_{p,\alpha} u, |u|^{p-2} \, u \, \one_{[u \neq 0]}) \geq 0$ in this case.
\\
{\bf Case 2:} Suppose $R_a \neq 0$.
Expanding \eqref{P} using Lemma \ref{C alpha identities} gives
\begin{eqnarray}
P 
&=& (R_{s,\alpha} \, \xi', \xi') + (R_{s,\alpha} \, \eta, \eta) 
	+ \frac{p}{\sqrt{p-1}} \, (B_{a,\alpha} \, \xi', \eta)
	+ \frac{p-2}{\sqrt{p-1}} \, (B_{s,\alpha} \, \xi', \eta)
\nonumber
\\ 
&=& (\cos\alpha) \, \big( (R_s \, \xi', \xi') + (R_s \, \eta, \eta) \big)
	- (\sin\alpha) \, \big( (B_s \, \xi', \xi') + (B_s \, \eta, \eta) \big)
	\nonumber
	\\
	&& {} + \frac{p}{\sqrt{p-1}} \, (\sin\alpha) \, (R_a \, \xi', \eta)
			+ \frac{p-2}{\sqrt{p-1}} \, (\sin\alpha) \, (R_s \, \xi', \eta) 
			+ \frac{p-2}{\sqrt{p-1}} \, (\cos\alpha) \, (B_s \, \xi', \eta)
\nonumber
\\
&=& (\cos\alpha) \, \big( (R_s \, \xi', \xi') + (R_s \, \eta, \eta) \big)
	- (\sin\alpha) \, \big( (B_s \, \xi', \xi') + (B_s \, \eta, \eta) - 2 \, (R_a \, \xi', \eta) \big)
	\nonumber
	\\
	&& {} + \big( \frac{p}{\sqrt{p-1}} - 2 \big) \, (\sin\alpha) \, (R_a \, \xi', \eta)
			+ \frac{p-2}{\sqrt{p-1}} \, (\sin\alpha) \, (R_s \, \xi', \eta) 
			\nonumber
	\\
	&& {} + \frac{p-2}{\sqrt{p-1}} \, (\cos\alpha) \, (B_s \, \xi', \eta),
\label{P terms}
\end{eqnarray}
where we used Lemma \ref{C alpha identities}\ref{R alpha} and \ref{B alpha} in the second step.
Next we estimate the terms in \eqref{P terms}.
By Lemma \ref{sector equiv form} we have
\begin{equation} \label{P 2nd term}
\Big|
(\sin\alpha) \, \big( (B_s \, \xi', \xi') + (B_s \, \eta, \eta) - 2 \, (R_a \, \xi', \eta) \big)
\Big|
\leq 
\sin(|\alpha|) \, (\tan\theta) \, \big( (R_s \, \xi', \xi') + (R_s \, \eta, \eta) \big)
\end{equation}
since $B_a = 0$ by hypothesis.
Using Lemma \ref{Ra < Rs} and the fact that $\sin\phi = \frac{2 \, \sqrt{p-1}}{p}$ we deduce that
\begin{equation} \label{P 3rd term}
\Big| \big( \frac{p}{\sqrt{p-1}} - 2 \big) \, (\sin\alpha) \, (R_a \, \xi', \eta) \Big|
\leq 
\big( \frac{2}{\sin\phi} - 2 \big) \, \sin(|\alpha|) \, (\tan\theta) \, 
	\big( (R_s \, \xi', \xi') + (R_s \, \eta, \eta) \big).
\end{equation}
Next note that $\cot\phi = \frac{|p-2|}{2 \, \sqrt{p-1}}$.
Therefore 
\begin{equation} \label{P 4th term}
\Big| \frac{p-2}{\sqrt{p-1}} \, (\sin\alpha) \, (R_s \, \xi', \eta) \Big|
\leq 
(\cot\phi) \, \sin(|\alpha|) \, \big( (R_s \, \xi', \xi') + (R_s \, \eta, \eta) \big)
\end{equation}
by Lemma \ref{Rs}.
It follows from Lemma \ref{W Bs < Rs} that
\begin{equation} \label{P 5th term}
\Big| \frac{p-2}{\sqrt{p-1}} \, (\cos\alpha) \, (B_s \, \xi', \eta) \Big|
\leq (\cot\phi) \, (\cos\alpha) \, (\tan\theta) \, 
	\big( (R_s \, \xi', \xi') + (R_s \, \eta, \eta) \big).
\end{equation}
Next \eqref{P terms}, \eqref{P 2nd term}, \eqref{P 3rd term}, \eqref{P 4th term} and \eqref{P 5th term} together imply
\begin{eqnarray}
P
&\geq& \left( \big( 1 - (\tan\theta) \, \cot\phi \big) \, \cos\alpha
	- \Big( (\frac{2}{\sin\phi} - 1) \, \tan\theta + \cot\phi \Big) \, \sin(|\alpha|) \right) \,
	\big( (R_s \, \xi', \xi') + (R_s \, \eta, \eta) \big)
\nonumber
\\
&\geq& 0,
\label{P positive}
\end{eqnarray}
where we used that fact that $\alpha \in (-\psi, \psi)$ and Lemma \ref{Rs} in the last step.
Combining \eqref{int P} and \eqref{P positive} yields $\R (B_{p,\alpha} u, |u|^{p-2} \, u \, \one_{[u \neq 0]}) \geq 0$ in this case.
\end{proof}

Next we prove the second inequality for $B_{p,\alpha}$.
We need the following density result.

\begin{prop} \label{W Cc dense 1}
Let $\alpha \in (-\psi, \psi)$, where $\psi$ is given by \eqref{psi}.
Then the space $C_c^\infty(\Ri^d)$ is dense in $(D(B_{p,\alpha}) \cap W^{1,p}(\Ri^d), \|\cdot\|_{D(B_{p,\alpha})})$.
\end{prop}

\begin{proof}
The claim follows from \cite[Proposition 4.7]{Do1} (see also \cite[Proposition 4.23]{Do}).
\end{proof}

%
%
%

The second inequality is as follows (see \cite[proposition 6.1]{WongDzung} for the case when $\alpha = 0$ and $B_p$ has real symmetric coefficients as well as \cite[Proposition 4.8]{Do1} for the case when $\alpha = 0$).

\begin{prop} \label{W 2nd ineq}
Suppose $B_a = 0$.
Let $p \in (1, \infty)$ be such that $|1 - \frac{2}{p}| < \cos \theta$.
Let $\alpha \in (-\gamma, \gamma)$, where $\gamma$ is given by \eqref{gamma}.
Then there exists an $M > 0$ such that 
\[
\R (\nabla (B_{p,\alpha} u), 
	|\nabla u|^{p-2} \, \nabla u \, \one_{[\nabla u \neq 0]}) 
\geq - M \, \|\nabla u\|_p^p
\]
for all $u \in W^{2,p}(\Ri^d)$ such that $\nabla(B_{p,\alpha} u) \in (L_p(\Ri^d))^d$.
\end{prop}

\begin{proof}
We consider two cases.
\\
{\bf Case 1:} Suppose $R_a = 0$.
Then it follows from Lemma \ref{C alpha identities} that $B_{a,\alpha} = R_a \, \sin\alpha = 0$. 
Moreover, the condition $\alpha \in (-\psi, \psi)$ implies $\tan(\theta+|\alpha|) < \tan\phi$.
Therefore \cite[Proposition 4.8]{Do1} still applies to yield the result.
\\
{\bf Case 2:} Suppose $R_a \neq 0$. 
If $\alpha = 0$, the claim follows from \cite[Proposition 4.8]{Do1}.
Therefore we may assume that $\alpha \neq 0$ for the rest of the proof.
Note that $\alpha \in (-\gamma, \gamma)$ implies $(\tan\theta) \, \tan(|\alpha|) < \frac{1}{3}$ and $K \, \tan(|\alpha|) < 1$, where $K$ is defined by \eqref{K value}.
Let $\varepsilon_0 \in (0, 1 \wedge (p-1))$ be such that 
\begin{equation} \label{tan alpha 1st ineq}
(\tan\theta) \, \tan(|\alpha|) \leq \frac{1 - \varepsilon}{3 - \varepsilon}
\end{equation}
and
\begin{eqnarray} 
&& \left( \Big( \frac{p}{\sqrt{(1-\varepsilon)(p-1-\varepsilon)}} - 1 \Big) \, \tan\theta 
+ \frac{|p-2|}{2 \, \sqrt{(1-\varepsilon)(p-1-\varepsilon)}} \right) \, \tan(|\alpha|) 
\nonumber
\\
&\leq&
1 - (\tan\theta) \, \frac{|p-2|}{2 \, \sqrt{(1-\varepsilon)(p-1-\varepsilon)}}
\label{tan alpha 2nd ineq}
\end{eqnarray}
for all $\varepsilon \in (0, \varepsilon_0)$.
Let $\varepsilon \in (0, \varepsilon_0)$ be such that 
\begin{equation} \label{W correct varepsilon}
\varepsilon 
< \frac{\varepsilon_0}{32 \, d \, \big( 1 + \tan(\theta + |\alpha|) \big)^2 \, 
	\sup_{1 \leq l \leq d} \|\D_l^2 C\|_\infty}.
\end{equation}

Let $u \in W^{2,p}(\Ri^d)$.
By Lemma \ref{W Cc dense 1} we can assume without loss of generality that $u$ has a compact support.
For the rest of the proof, all integrations are over the set $\{x \in \Ri^d: |(\nabla u)(x)| \neq 0 \}$.
We have
\begin{eqnarray*}
(\nabla (B_{p,\alpha} u), |\nabla u|^{p-2} \, \nabla u)
& = & - \sum_{k,l,j=1}^d \int \Big( \D_j \D_l (e^{i \alpha} \, c_{kl} \, \D_k u) \Big) \, 
	|\nabla u|^{p-2} \, \D_j \overline{u}
\\
& = & - \sum_{k,l,j=1}^d \int e^{i \alpha} \, \Big( \D_l \big( (\D_j c_{kl}) \, (\D_k u) 
	+ c_{kl} \, (\D_j \D_k u) \big) \Big) \,
	|\nabla u|^{p-2} \, \D_j \overline{u}
\\
& = & - \sum_{k,l,j=1}^d \int e^{i \alpha} \, \Big( \D_l \big( (\D_j c_{kl}) \, (\D_k u) \big) \Big) \, 
	|\nabla u|^{p-2} \, \D_j \overline{u}
\\*
& & {} + \sum_{k,l,j=1}^d \int e^{i \alpha} \, c_{kl} \, (\D_j \D_k u) \, 
	\D_l \big( |\nabla u|^{p-2} \, \D_j \overline{u} \big)
\\
& = & ({\rm I}) + ({\rm II}).
\end{eqnarray*}
We first consider the real part of (I).
We have 
\begin{eqnarray*}
&& -\R \sum_{k,l,j=1}^d \int e^{i \alpha} \, \Big( \D_l \big( (\D_j c_{kl}) \, (\D_k u) \big) \Big) \, 
	|\nabla u|^{p-2} \, \D_j \overline{u}
\\
&=& - \R \sum_{k,l,j=1}^d \int e^{i \alpha} \, (\D_l \D_j c_{kl}) \, (\D_k u) \, 
	(\D_j \overline{u}) \, |\nabla u|^{p-2}
\\
&& {} - \R \sum_{k,l,j=1}^d \int e^{i \alpha} \, (\D_j c_{kl}) \, (\D_l \D_k u) \, 
	(\D_j \overline{u}) \, |\nabla u|^{p-2}
\\
& = & ({\rm Ia}) + ({\rm Ib}).
\end{eqnarray*}
For (Ia) we have
\[
({\rm Ia}) 
\geq - \frac{1}{2} \sum_{k,l,j=1}^d \|c_{kl}\|_{W^{2,\infty}} 
	\int (|\D_k u|^2 + |\D_j u|^2) \, |\nabla u|^{p-2}
\geq - M_1 \, \|\nabla u\|_p^p,
\]
where $M_1 = d^2 \, \sup \{\|c_{kl}\|_{W^{2,\infty}}: 1 \leq k,l \leq d\}$.
Let $U = (\D_l \D_k u)_{1 \leq k,l \leq d}$.
For (Ib) we estimate
\begin{eqnarray*}
({\rm Ib})  
& = & - \R \sum_{j=1}^d \int \tr((\D_j C_\alpha) \, U) \, (\D_j \overline{u}) \, |\nabla u|^{p-2}
\\
& \geq & - \sum_{j=1}^d \int \Big( \varepsilon \, |\tr((\D_j C_\alpha) \, U)|^2 \, |\nabla u|^{p-2}
	+ \frac{1}{4 \varepsilon} \, |\D_j \overline{u}|^2 \, |\nabla u|^{p-2} \Big)
\\
& \geq & - \varepsilon' \int \tr(U \, R_{s,\alpha} \, \overline{U}) \, |\nabla u|^{p-2}
	- M_2 \, \|\nabla u\|_p^p
\\
& = & - \varepsilon' \int \tr(\overline{U} \, R_{s,\alpha} \, U) \, |\nabla u|^{p-2}
	- M_2 \, \|\nabla u\|_p^p,
\end{eqnarray*}
where we used Lemma \ref{W Oleinik 2} in the third step with 
\[
\varepsilon' 
= 32 \, \varepsilon \, d \, \big( 1 + \tan(\theta + |\alpha|) \big)^2 \, 
	\sup_{1 \leq l \leq d} \|\D_l^2 C\|_\infty
\]
and $M_2 = \frac{1}{4 \varepsilon}$.
Note that $\varepsilon' \in (0, \varepsilon_0)$ by \eqref{W correct varepsilon}.

Next we consider the real part of (II).
Note that 
\begin{eqnarray*}
\R \sum_{k,l,j=1}^d \int e^{i \alpha} \, c_{kl} \, (\D_j \D_k u) \, 
	\D_l \big( |\nabla u|^{p-2} \, \D_j \overline{u} \big)
& = & \R \sum_{k,l,j=1}^d \int e^{i \alpha} \, c_{kl} \, (\D_j \D_k u) \, (\D_l \D_j \overline{u}) \, |\nabla u|^{p-2}
\\
& & {} + \R \sum_{k,l,j=1}^d \int e^{i \alpha} \, c_{kl} \, (\D_j \D_k u) \, 
	(\D_j \overline{u}) \, \D_l(|\nabla u|^{p-2})
\\
& = & ({\rm IIa}) + ({\rm IIb}).
\end{eqnarray*}
In what follows we let $U \, \nabla \overline{u} = \xi + i \, \eta$, where $\xi, \eta \in \Ri^d$.
For (IIa) we have 
\[
({\rm IIa}) 
= \int \tr(\overline{U} \, \R (C_\alpha) \, U) \, |\nabla u|^{p-2}
= \int \tr(\overline{U} \, R_{s,\alpha} \, U) \, |\nabla u|^{p-2} 
	+ i \int \tr(\overline{U} \, B_{a,\alpha} \, U) \, |\nabla u|^{p-2}.
\]
For (IIb) we have
\begin{eqnarray*}
({\rm IIb}) 
& = & \R \sum_{k,l,i,j=1}^d \frac{p-2}{2} \int e^{i \alpha} \, c_{kl} \, (\D_j \D_k u) \, (\D_j \overline{u}) \, 
	\Big( (\D_l \D_i u) \, (\D_i \overline{u}) + (\D_l \D_i \overline{u}) \, (\D_i u) \Big) \, |\nabla u|^{p-4}
\\
& = & \frac{p-2}{2} \int \R \Big( \big( C_\alpha \, U \, \nabla \overline{u}, \overline{U \, \nabla \overline{u}} \big)
	+ \big( C_\alpha \, U \, \nabla \overline{u}, U \, \nabla \overline{u} \big) \Big) \, |\nabla u|^{p-4}
\\
& = & (p-2) \int \Big( (R_\alpha \, \xi, \xi) - (B_\alpha \, \eta, \xi) \Big) \, |\nabla u|^{p-4}
\\
& = & (p-2) \int \Big( (R_{s,\alpha} \, \xi, \xi) - (B_{s,\alpha} \, \xi, \eta) + (B_{a,\alpha} \, \xi, \eta) \Big) \, |\nabla u|^{p-4},
\end{eqnarray*}
where $\xi, \eta \in \Ri^d$ and $U \, \nabla \overline{u} = \xi + i \, \eta$.

In total we obtain
\begin{eqnarray}
\R (\nabla (B_{p,\alpha} u), |\nabla u|^{p-2} \, \nabla u)
& \geq & - (M_1 + M_2) \, \|\nabla u\|_p^p
	+ (1-\varepsilon') \int \tr(U \, R_{s,\alpha} \, \overline{U}) \, |\nabla u|^{p-2}
\nonumber
\\
& & {} + i \int \tr(U \, B_{a,\alpha} \, \overline{U}) \, |\nabla u|^{p-2}
\nonumber
\\
& & 
	{} + (p-2) \int \Big( (R_{s,\alpha} \, \xi, \xi) - (B_{s,\alpha} \, \xi, \eta) + (B_{a,\alpha} \, \xi, \eta) \Big) \, |\nabla u|^{p-4}
\nonumber
\\
& = & - (M_1 + M_2) \, \|\nabla u\|_p^p + P,
\label{2nd ineq intermediate}
\end{eqnarray}
where 
\begin{eqnarray*}
P 
&=& (1-\varepsilon') \int \tr(U \, R_{s,\alpha} \, \overline{U}) \, |\nabla u|^{p-2}
	+ i \int \tr(U \, B_{a,\alpha} \, \overline{U}) \, |\nabla u|^{p-2}
	\\
	&& {} + (p-2) \int \Big( (R_{s,\alpha} \, \xi, \xi) - (B_{s,\alpha} \, \xi, \eta) 
		+ (B_{a,\alpha} \, \xi, \eta) \Big) \, |\nabla u|^{p-4}.
\end{eqnarray*}
Next we will show that $P \geq 0$.
First note that $(1-\varepsilon') \, (\cos\alpha) - (3 - \varepsilon') \, \sin(|\alpha|) \, \tan\theta \geq 0$ due to \eqref{tan alpha 1st ineq}.
It follows that
\begin{eqnarray*}
&& (1-\varepsilon') \tr(\overline{U} \, R_{s,\alpha} \, U) \, |\nabla u|^2
	+ i \, \tr(\overline{U} \, B_{a,\alpha} \, U) \, |\nabla u|^2
\\
&=& (1-\varepsilon') \, (\cos\alpha) \, \tr(\overline{U} \, R_s \, U) \, |\nabla u|^2
	- (1-\varepsilon') \, (\sin\alpha) \, \tr(\overline{U} \, B_s \, U) \, |\nabla u|^2
	\\
	&& {} + i \, (\sin\alpha) \, \tr(\overline{U} \, R_a \, U) \, |\nabla u|^2
\\
&=& \Big( (1-\varepsilon') \, \cos\alpha - (3 - \varepsilon') \, \sin(|\alpha|) \, \tan\theta \Big) \, 
	\tr(\overline{U} \, R_s \, U) \, |\nabla u|^2
	\\
	&& {} + (1-\varepsilon') \, \sin(|\alpha|) \, \tr \left( \overline{U} \, \Big( (\tan\theta) \, R_s 
		- \frac{\sin\alpha}{\sin(|\alpha|)} B_s \Big) \, U \right) \, |\nabla u|^2
	\\
	&& {} + \sin(|\alpha|) \, \tr \left( \overline{U} \, \Big( 2 \, (\tan\theta) \, R_s 
		+ i \, \frac{\sin\alpha}{\sin(|\alpha|)} R_a \Big) \, U \right) \, |\nabla u|^2
\\
&\geq& \Big( (1-\varepsilon') \, (\cos\alpha) - (3 - \varepsilon') \, \sin(|\alpha|) \, \tan\theta \Big) \, 
	(R_s \, U \, \nabla \overline{u}, U \, \nabla \overline{u})
	\\
	&& {} + (1-\varepsilon') \, \sin(|\alpha|) \, \Big( \big( (\tan\theta) \, R_s 
		- \frac{\sin\alpha}{\sin(|\alpha|)} B_s \big) \, U \, \nabla \overline{u}, U \, 
		\nabla \overline{u} \Big)
	\\
	&& {} + \sin(|\alpha|) \, \Big( \big( 2 \, (\tan\theta) \, R_s 
		+ i \, \frac{\sin\alpha}{\sin(|\alpha|)} R_a \big) \, U \, \nabla \overline{u}, U \, 
		\nabla \overline{u} \Big)
\\
&=& (1-\varepsilon') \, (\cos\alpha) \, (R_s \, U \, \nabla \overline{u}, U \, \nabla \overline{u})
	- (1-\varepsilon') \, (\sin\alpha) \, (B_s \, U \, \nabla \overline{u}, U \, \nabla \overline{u})
	\\
	&& {} + i \, (\sin\alpha) \, (R_a \, U \, \nabla \overline{u}, U \, \nabla \overline{u})
\\
&=& (1-\varepsilon') \, (\cos\alpha) \, \Big( (R_s \, \xi, \xi) + (R_s \, \eta, \eta) \Big)
	- (1-\varepsilon') \, (\sin\alpha) \, \Big( (B_s \, \xi, \xi) + (B_s \, \eta, \eta) \Big)
	\\
	&& {} + 2 \, (\sin\alpha) \, (R_a \, \xi, \eta),
\end{eqnarray*}
where we used Lemmas \ref{W R<trR} and \ref{positive definite} in the third step.
Hence we obtain
\begin{eqnarray}
P
&\geq& \int \bigg( (1-\varepsilon') \, (\cos\alpha) \, \Big( (R_s \, \xi, \xi) + (R_s \, \eta, \eta) \Big)
	- (1-\varepsilon') \, (\sin\alpha) \, \Big( (B_s \, \xi, \xi) + (B_s \, \eta, \eta) \Big)
	\nonumber
	\\
	&& {} + 2 \, (\sin\alpha) \, (R_a \, \xi, \eta) \bigg) \, |\nabla u|^{p-4}
	\nonumber
	\\
	&& {} + (p-2) \int \Big( (R_{s,\alpha} \, \xi, \xi) - (B_{s,\alpha} \, \xi, \eta) 
		+ (B_{a,\alpha} \, \xi, \eta) \Big) \, |\nabla u|^{p-4}
\nonumber
\\
&=& \int \bigg( (\cos\alpha) \, \big( (p-1-\varepsilon') \, (R_s \, \xi, \xi) 
		+ (1-\varepsilon') \, (R_s \, \eta, \eta) \big)
	\nonumber
	\\
	&& {} - (\sin\alpha) \, \big( (p-1-\varepsilon') \, (B_s \, \xi, \xi) 
		+ (1-\varepsilon') \, (B_s \, \eta, \eta) \big)
	\nonumber
	\\
	&& {} + p \, (\sin\alpha) \, (R_a \, \xi, \eta) - (p-2) \, (\sin\alpha) \, (R_s \, \xi, \eta) 
		- (p-2) \, (\cos\alpha) \, (B_s \, \xi, \eta) \bigg) \, |\nabla u|^{p-4}
\nonumber
\\
&=& \int \bigg( (\cos\alpha) \, \big( (R_s \, \xi', \xi') + (R_s \, \eta', \eta') \big)
		 - (\sin\alpha) \, \big( (B_s \, \xi', \xi') + (B_s \, \eta', \eta') \big)
	\nonumber
	\\
	&& {} + \frac{p}{\sqrt{(1-\varepsilon') \, (p-1-\varepsilon')}} \, (\sin\alpha) \, (R_a \, \xi', \eta') 
		- \frac{p-2}{\sqrt{(1-\varepsilon') \, (p-1-\varepsilon')}} \, (\sin\alpha) \, (R_s \, \xi', \eta') 
	\nonumber
	\\
	&& {} - \frac{p-2}{\sqrt{(1-\varepsilon') \, (p-1-\varepsilon')}} \, (\cos\alpha) \, (B_s \, \xi', \eta') \bigg) \, |\nabla u|^{p-4},
\label{P positive intermediate}
\end{eqnarray}
where we used Lemma \ref{C alpha identities}\ref{R alpha} and \ref{B alpha} in the second step, $\xi' = \sqrt{p-1-\varepsilon'} \, \xi$ and $\eta' = \sqrt{1-\varepsilon'} \, \eta$.
Finally using \eqref{tan alpha 2nd ineq} we argue in a similar manner to that used in Case 2 of the proof of Proposition \ref{W 1st ineq} to derive $P \geq 0$.
Thus it follows from \eqref{2nd ineq intermediate} that
\[
\R (\nabla (B_{p,\alpha} u), |\nabla u|^{p-2} \, \nabla u) \geq - (M_1 + M_2) \, \|\nabla u\|_p^p
\]
as claimed.
\end{proof}

Next we use the two inequalities obtained in Propositions \ref{W 1st ineq} and \ref{W 2nd ineq} to show that $B_{p,\alpha}$ is $m$-accretive for all $\alpha$ in a suitable range.

\begin{prop} \label{W Cc dense 2}
Suppose $B_a = 0$.
Let $p \in (1, \infty)$ be such that $|1 - \frac{2}{p}| < \cos \theta$.
Let $\alpha \in (-\gamma, \gamma)$, where $\gamma$ is given by \eqref{gamma}.
Then $B_{p,\alpha}$ is $m$-accretive.
\end{prop}

\begin{proof}
The result follows from the arguments used in the proof of \cite[Proposition 4.9]{Do1}.
Note that \cite[Propositions 4.1, 4.7 and 4.8]{Do1} used in the proof of \cite[Proposition 4.9]{Do1} are now replaced by Propositions \ref{W 1st ineq}, \ref{W Cc dense 1} and \ref{W 2nd ineq} respectively.
\end{proof}

We now have enough preliminary results to prove Theorem \ref{contraction holomorphic semigroup}.

\begin{proof}[{\bf Proof of Theorem \ref{contraction holomorphic semigroup}}]
We consider two parts.
\\
(i) Contractivity: 
Using Proposition \ref{W Cc dense 2} and \cite[Theorem IX.1.23]{Kat1} we deduce that $-B_p$ generates a holomorphic semigroup with angle $\psi$ given by \eqref{psi} which is contractive on the sector $\Sigma_\gamma$, where $\gamma$ is given by \eqref{gamma}.
Note that $B_p = \overline{A_p}$ by Proposition \ref{Ap m accretive}.
Hence $S^{(p)}$ is contractive on $\Sigma_\gamma$.
\\
(ii) Consistency: 
It suffices to show that $S^{(p)}$ is consistent with $S$.
It follows from \cite[Propositions 1.1 and 5.1]{Do1} that the $C_0$-semigroup generated by $-B_2$ is consistent with the $C_0$-semigroup generated by $-B_p$.
Since $B_2 = \overline{A_2}$ and $B_p = \overline{A_p}$ by Proposition \ref{Ap m accretive}, the semigroup $S^{(p)}$ is consistent with $S$ as required.
\end{proof}

\subsection*{Acknowledgements}
I wish to thank Tom ter Elst for giving detailed and valuable comments.


\begin{thebibliography}{Ouh05}

\bibitem[CM05]{CiaM}
{\sc Cialdea, A. {\rm and} Maz'ya, V.}, Criterion for the $L^p$-dissipativity
  of second order differential operators with complex coefficients.
\newblock {\em J. Math. Pures Appl. (9)} {\bf 84} (2005),  1067--1100.

\bibitem[Dav89]{Dav2}
{\sc Davies, E.~B.}, {\em Heat kernels and spectral theory}.
\newblock Cambridge Tracts in Mathematics 92. Cambridge University Press,
  Cambridge etc., 1989.

\bibitem[Do16a]{Do1}
{\sc Do, T.~D.}, Degenerate elliptic operators in $L_p$-spaces with complex
  $W^{2,\infty}$-coefficients.
\newblock Submitted.

\bibitem[Do16b]{Do}
\leavevmode\vrule height 2pt depth -1.6pt width 23pt, {\em Degenerate elliptic
  second-order differential operators with bounded complex-valued
  coefficients}.
\newblock PhD thesis, The University of Auckland, New Zealand, 2016.

\bibitem[DE16]{DE1}
{\sc Do, T.~D. {\rm and} Elst, A. F.~M. ter}, One-dimensional degenerate
  elliptic operators on $L_p$-spaces with complex coefficients.
\newblock {\em Semigroup Forum} {\bf 92} (2016),  559--586.

\bibitem[Epp89]{Epp}
{\sc Epperson, J.~B.}, The hypercontractive approach to exactly bounding an
  operator with complex {G}aussian kernel.
\newblock {\em J. Funct. Anal.} {\bf 87} (1989),  1--30.

\bibitem[Gol85]{Gol}
{\sc Goldstein, J.~A.}, {\em Semigroups of linear operators and applications}.
\newblock Oxford Mathematical Monographs. The Clarendon Press, Oxford
  University Press, New York, 1985.

\bibitem[Kat80]{Kat1}
{\sc Kato, T.}, {\em Perturbation theory for linear operators}.
\newblock Second edition, Grund\-lehren der mathematischen Wissenschaften 132.
  Springer-Verlag, Berlin etc., 1980.

\bibitem[LP95]{LiP1}
{\sc Liskevich, V.~A. {\rm and} Perelmuter, M.~A.}, Analyticity of
  sub-{M}arkovian semigroups.
\newblock {\em Proc. Amer. Math. Soc.} {\bf 123} (1995),  1097--1104.

\bibitem[MS08]{MS}
{\sc Metafune, G. {\rm and} Spina, C.}, An integration by parts formula in
  Sobolev spaces.
\newblock {\em Mediterr. J. Math.} {\bf 5} (2008),  357--369.

\bibitem[Oka91]{Oka2}
{\sc Okazawa, N.}, Sectorialness of second order elliptic operators in
  divergence form.
\newblock {\em Proc. Amer. Math. Soc.} {\bf 113} (1991),  701--706.

\bibitem[Ouh05]{Ouh5}
{\sc Ouhabaz, E.-M.}, {\em Analysis of heat equations on domains}, vol.\ 31 of
  London Mathematical Society Monographs Series.
\newblock Princeton University Press, Princeton, NJ, 2005.

\bibitem[Paz83]{Paz}
{\sc Pazy, A.}, {\em Semigroups of linear operators and applications to partial
  differential equations}.
\newblock Applied mathematical sciences 44. Springer-Verlag, New York etc.,
  1983.

\bibitem[RS75]{RS2}
{\sc Reed, M. {\rm and} Simon, B.}, {\em Methods of modern mathematical physics
  II. Fourier analysis, self-adjoint\-ness}.
\newblock Academic Press, New York etc., 1975.

\bibitem[Ste70]{Ste2}
{\sc Stein, E.~M.}, {\em Topics in harmonic analysis related to the
  Littlewood--Paley theory}.
\newblock Annals of Mathematics Studies 63. Princeton University Press,
  Princeton, 1970.

\bibitem[WD83]{WongDzung}
{\sc Wong-Dzung, B.}, $L^p$-Theory of degenerate-elliptic and parabolic
  operators of second order.
\newblock {\em Proc. Roy. Soc. Edinburgh Sect. A} {\bf 95} (1983),  95--113.

\bibitem[Wei]{Wei}
{\sc Weissler, F.~B.}, Two-point inequalities, the {H}ermite semigroup, and the
  {G}auss-{W}eierstrass semigroup.
\newblock {\em J. Funct. Anal.} {\bf 32},  102.

\end{thebibliography}

\end{document}